\documentclass{amsart}
\usepackage{amssymb,amscd}

\newtheorem{lemma}{Lemma}
\newtheorem{prop}[lemma]{Proposition}
\newtheorem{cor}[lemma]{Corollary}
\newtheorem{thm}[lemma]{Theorem}

\newtheorem{thm?}[lemma]{Theorem?}
\newtheorem{ques}{Question}

\newtheorem*{nolabelthm}{Main Theorem}

\DeclareFontEncoding{OT2}{}{} 
  \newcommand{\textcyr}[1]{%
    {\fontencoding{OT2}\fontfamily{wncyr}\fontseries{m}\fontshape{n}%
     \selectfont #1}}
\newcommand{\Sha}{{\mbox{\textcyr{Sh}}}}

\title{Period-index problems in WC-groups IV: a local transition theorem}
\author{Pete L. Clark}
\thanks{The author is partially supported by National Science Foundation grant DMS-0701771.}
\address{Department of Mathematics \\ Boyd Graduate Studies Research Center \\ University 
of Georgia \\ Athens, GA 30602-7403 \\ USA}
\email{plclark@gmail.com}

\newcommand{\F}{\ensuremath{\mathbb F}}

\newcommand{\N}{\ensuremath{\mathbb N}}
\newcommand{\Q}{\ensuremath{\mathbb Q}}
\newcommand{\R}{\ensuremath{\mathbb R}}
\newcommand{\Z}{\ensuremath{\mathbb Z}}

\newcommand{\C}{\ensuremath{\mathbb C}}

\newcommand{\ra}{\ensuremath{\rightarrow}}

\newcommand{\Hom}{\operatorname{Hom}}

\newcommand{\Ker}{\operatorname{Ker}}

\newcommand{\unr}{\operatorname{unr}}

\newcommand{\PP}{\ensuremath{\mathbb P}}

\newcommand{\Aut}{\operatorname{Aut}}

\newcommand{\GL}{\operatorname{GL}}

\newcommand{\red}{\operatorname{red}}

\renewcommand{\Im}{\operatorname{Im}}

\newcommand{\Pic}{\operatorname{Pic}}
\newcommand{\FPic}{\mathbf{Pic}}
\newcommand{\Gm}{\mathbb G_m}

\newcommand{\gk}{\mathfrak{g}_k}

\newcommand{\sep}{\operatorname{sep}}

\newcommand{\Res}{\operatorname{Res}}
\newcommand{\Br}{\operatorname{Br}}
\newcommand{\et}{\operatorname{\'et}}
\newcommand{\Spec}{\operatorname{Spec}}

\newcommand{\gK}{\mathfrak{g}_K}

\newcommand{\car}{\operatorname{char}}
\renewcommand{\Aut}{\operatorname{Aut}}
\newcommand{\rank}{\operatorname{rank}}
\newcommand{\cd}{\operatorname{cd}}
\renewcommand{\gg}{\mathfrak{g}}
\renewcommand{\et}{\operatorname{et}}
\newcommand{\fl}{\operatorname{fl}}
\newcommand{\PGL}{\operatorname{PGL}}
\newcommand{\PS}{\operatorname{PS}}
\newcommand{\reg}{\operatorname{reg}}

\begin{document}
\maketitle

\begin{abstract}
Let $K$ be a complete discretely valued field with perfect residue field $k$.  Assuming upper bounds on the relation between period and index for WC-groups over $k$, 
we deduce corresponding upper bounds on the relation between period and index for WC-groups over $K$.  Up to a constant 
depending only on the dimension of the torsor, we recover theorems of 
Lichtenbaum and Milne in a ``duality free'' context.  Our techniques include the use of \emph{LLR models} of torsors under 
abelian varieties with good reduction and a generalization of the period-index 
obstruction map to flat cohomology.  In an appendix, we consider some related issues of a field-arithmetic nature.
\end{abstract}


\noindent

\section*{Introduction}

\subsection{Notation and Terminology}
\textbf{} \\ \\ \noindent
For a field $K$, we let $K^{\sep}$ denote a separable closure of $K$ and $\overline{K}$ an algebraic closure of $K$.  We 
write $\mathfrak{g}_K$ for $\Aut(K^{\sep}/K)$.  
\\ \\
If $X_{/K}$ is an integral variety, let $X^{\reg}$ denote its regular locus and $I(X)$ the index of $X$, i.e., the gcd of 
all degrees of closed points on $X$.
\\ \\ 
For $M$ a $\mathfrak{g}_K$-module and $\eta \in H^i(K,M)$ a Galois cohomology class, we denote by $P(\eta)$ and $I(\eta)$ 
the \textbf{period} and \textbf{index} of $\eta$ (c.f. \cite[$\S 2$]{WCII}).  Especially, if $A_{/K}$ is an abelian variety, then 
$H^1(K,A)$ is canonically isomorphic to the \textbf{Weil-Ch\^atelet group} of $A$, which parameterizes \textbf{torsors} 
$(X,\mu)$ under $A$.  Under this correspondence, we have $I(\eta) = I(X)$.
\\ \\
Let $G_{/K}$ be an algebraic group scheme.  Then $G$ gives rise to a sheaf of groups on the flat site of $\Spec K$.  We put 
$H^0(K,G) = G(K)$, and we denote by $H^1(K,G)$ the first flat cohomology, a pointed set.  If $G$ is commutative, for all $i \geq 0$ we have flat 
cohomology groups $H^i(K,G)$.  No confusion should arise between flat and \'etale (= Galois, here) cohomology, because flat 
and \'etale cohomology coincide when $G_{/K}$ is a smooth group scheme, whereas if $G_{/K}$ is not smooth, strictly 
speaking Galois cohomology of $G$ does not make sense.  (Only once do we consider in passing the Galois cohomology of 
the maximal \'etale quotient of a non-smooth finite $K$-group scheme $G$, and then we use the notation $G^{\circ}$ to 
disinguish this quotient from $G$ itself.)
\\ \\
Recall that a principal polarization on an abelian variety $A_{/K}$ is a $K$-rational element $\lambda$ of the N\'eron-Severi group 
$NS(A)$ such that the corresponding homomorphism $\varphi_{\lambda}: A \ra A^{\vee}$ has the property that 
$(\varphi_{\lambda})_{/K^{\sep}} = 
\varphi_{\mathcal{L}}$ for some ample line bundle $\mathcal{L} \in \Pic(A_{/\overline{K}})$.  We say that a polarization 
$\lambda$ is \textbf{strong} if the line bundle $\mathcal{L}$ can be chosen to be $K$-rational.  The coboundary map in 
cohomology of the short exact sequence of $\gK$-modules
\[0 \ra \FPic^0(A)(K^{\sep}) \ra \FPic(A)(K^{\sep}) \ra NS(A)(K^{\sep}) \ra 0, \]
yields a homomorphism $\Phi_{\PS}: H^0(NS(A)) \ra H^1(K,A^{\vee})$ such that $\lambda \in NS(A)(K)$ is strong iff 
$\Phi_{\PS}(\lambda) = 0$.  
We recall from \cite[$\S 4$]{Poonen-Stoll} that \[\Phi_{\PS}(H^0(K,NS(A))) \subseteq H^1(K,A^{\vee})[2]. \]
It will then follow from Theorem \ref{GENERALPIBOUNDTHM} that any polarization 
can be made strong by passing to a field extension of degree at most $2^{2  \dim A}$.  It is easy to see that 
every polarization on an elliptic curve is strong.  

\subsection{The Main Theorem}

\begin{nolabelthm} 
\label{MT1}
\textbf{} \\ \noindent Let $K$ be a complete discretely valued field with perfect residue field $k$.  \\
a) Suppose that there exists $i \in \N$ and a function $c: \Z^+ \ra \Z^+$ such that: for all abelian varieties
$A_{/k}$ and all torsors $\eta \in H^1(k,A)$, 
\[ I(\eta) \leq c(\dim A) P(\eta)^i . \]
Then there exists a function $C: \Z^+ \ra \Z^+$ such that for all finite extensions $L/K$, all principally polarizable 
abelian varieties $A_{/L}$ and all torsors $\eta \in H^1(L,A)$, 
\[I(\eta) \leq C(\dim A) P(\eta)^{\dim A + i}. \]
b) Suppose that $\car(k) = 0$, that there exists $i \in \N$ and a function $c: \Z^+ \ra \Z^+$ such that: for all 
finite extensions $l/k$, all nontrivial abelian varieties $A_{/l}$ and all torsors $\eta \in H^1(l,A)$, we have 
\[I(\eta) \leq c(\dim A) P(\eta)^{\dim A + i - 1}. \]
Then there exists a function $C: \Z^+ \ra \Z^+$ such that for all finite extensions $L/K$, all principally polarizable 
abelian varieties $A_{/L}$ and all torsors $\eta \in H^1(L,A)$, 
\[I(\eta) \leq C(\dim A) P(\eta)^{\dim A + i}. \]
\end{nolabelthm}
\noindent
\subsection{Outline of the paper}
\textbf{} \\ \\ \noindent
The theorem as stated above is admittedly rather technical.  So we believe that most (if not all) readers will benefit from 
a discussion which places it in a larger context.
We do so at some length in $\S 1$, beginning in $\S 1.1$ by recalling some prior instances of 
``transition theorems'' in Galois cohomology and field arithmetic.  Throughout the rest of $\S 1$ we 
repeatedly give examples and remarks to show that many of the complications in the statement of the Main Theorem are necessary.
\\ \indent In $\S 2$ we introduce a new technical tool, the 
period-index obstruction map $\Delta$ in flat cohomology, which allows us to work with torsors with period divisible by 
the characteristic of the ground field, a case that was disallowed in our previous work on the subject.  As a first indication that our 
formalism is a fruitful one, we derive a foundational result (Theorem \ref{GENERALPIBOUNDTHM}) bounding the index of any 
torsor under an abelian variety $A$ in terms of the period and $\dim A$, which had up until now only been established in 
special cases. \\ \indent
The philosophy of the present work is to analyze the local period-index problem from geometric perspective, 
more precisely in terms of the geometry of regular $R$-models of torsors.  That we can proceed in this way is thanks in 
large part to two recent deep results on regular models due to Liu-Lorenzini-Raynaud and Gabber-Liu-Lorenzini.  We begin 
$\S 3$ with careful statements of these results (the latter of which is, as of this writing, not yet publicly available).  
The rest of the section is devoted to a proof of the Main Theorem as well as a more precise result (Theorem \ref{CTTHM2}) for 
iterated Laurent series fields over the complex numbers. 
\\ \indent
We have also included an Appendix, $\S 4$, which considers relations between the property of a field that all torsors 
under abelian varieties have rational points and some other, better known, properties of a field-arithmetic nature.

\subsection{Acknowledgements} \textbf{} \\ \\
\noindent This work was begun at the Mathematical Sciences 
Research Institute in 2006 and essentially completed at Universit\'e de Bordeaux I 
in 2009.  I am grateful to both institutions for providing financial support and 
effective working environments.  Thanks to Q. Liu and S. Sharif for helpful 
conversations.

\section{Motivating the Main Theorem}

\subsection{Transition theorems} \textbf{} \\ \\
\noindent
Recall the $C_i$ property of fields: a polynomial with coefficients in $K$ which is homogeneous polynomial 
of degree $n$ in more than $n^i$ variables has a nontrivial zero.  A field $C_0$ if and only if it is 
algebraically closed \cite[Lemma 3.2]{Greenberg}.  Moreover:
\begin{thm}
\label{CITHM}
a) (Chevalley) A finite field is $C_1$. \\
b) (Tsen) If $K$ is $C_i$, then so is any algebraic extension $L/K$. \\
c) (Tsen) If $K$ is $C_i$ and $L/K$ has transcendence degree $j$, then $L$ is $C_{i+j}$. \\
d) (Lang) A CDVF with algebraically closed residue field is $C_1$. \\
e) (Greenberg) If $k$ is $C_i$, then $k((t))$ is $C_{i+1}$.  
\end{thm}
\noindent
\begin{proof}
See \cite{Chevalley}, \cite{Tsen}, \cite{Lang52}, \cite{Greenberg67}.
\end{proof}
\noindent
Part c) can be rephrased as: if $K$ is a CDVF with $C_0$-residue field, then $K$ is 
$C_1$.  This suggests that Greenberg's theorem might be generalized to the statement that a 
CDVF $K$ with $C_i$ residue field is $C_{i+1}$.  In particular, E. Artin conjectured that 
$p$-adic fields are $C_2$.  But this turned out to be false: $\Q_p$ is not $C_i$ for any $i$.  
\\ \\
More recently other numerical invariants of fields with properties analogous to those of 
Theorem \ref{CITHM} have been considered.  Especially, in \cite{CG} J.-P. Serre defined, for a prime number $p$, the 
\textbf{$p$-cohomological dimension} $\cd_p(K)$ of a field $K$ as well as the \textbf{cohomological dimension} 
$\cd(K) = \sup_p \cd_p(K)$.  The relation $\cd_p(k) \leq i$ satisfies all the properties of Theorem 
\ref{CITHM}.  Moreover, a $C_1$ field satisfies $\cd_p(k) \leq 1$ for all $p$ and even the stronger 
property mentioned above: if $K$ is a CDVF with residue field $k$ satisfying $\cd_p(k) \leq i$, then 
$\cd_p(K) \leq i+1$ \cite[Prop. II.12]{CG}.  
\\ \\
In some cases there are relations between the cohomological dimension and rational points on 
certain $K$-varieties.  For a perfect field $K$, $\cd(K) \leq 1$ is equivalent 
to the condition that any torsor under a connected linear group has a $K$-rational point.  For a perfect field $K$, 
Serre conjectures that $\cd(K) \leq 2$ implies that every torsor under a simply connected, semisimple linear group has a $K$-rational point, which is now known to 
be true in (at least) many cases.  Finally, Voevodsky's work gives an interpretation of $\cd_2(K) \leq i$ in 
terms of the u-invariant of quadratic forms.     
\\ \\
There is a closely related problem on Brauer groups.  Say a field $K$ satisfies property
\textbf{Br(d)} if for all finite extensions $L/K$ and all $P \in \Z^+$, every class in $\Br(L)[P]$ 
has a splitting field of degree dividing $P^d$.  A perfect field $K$ is Br(0) -- 
that is, $\Br(L) = 0$ for all finite $L/K$ -- iff $\cd(K) \leq 1$.  Much of the 
content of class field theory is encoded in the statement that local and global 
fields are Br(1).  It is not hard to show that if $K$ is CDVF with perfect Br(i) residue field, then $K$ is Br(i+1).
On the other hand, the implication $K$ is Br(d) $\implies$ $K(t)$ is Br(d+1) is widely 
believed to be true but remains open, perhaps \emph{the} outstanding problem in the theory of Brauer groups.

\subsection{Property WC(i)}
\textbf{} \\ \\ \noindent
A field $K$ is \textbf{WC(i)} if for every abelian variety $A_{/K}$ and every $\eta \in H^1(K,A)[P]$, 
the index of $\eta$ divides $P^{i}$.  In particular, a field is WC$(0)$ iff 
every torsor under an abelian variety $A_{/K}$ has a $K$-rational point.
\\ \\
Remark 1.2.1: Let $L/K$ be a finite field extension, $A_{/L}$ an abelian variety, and $B = \Res_{L/K}$ be 
the Weil restriction of $B$ from $L$ down to $K$.  Shapiro's Lemma gives a canonical isomorphism 
$H^1(L,A) = H^1(K,B)$.  It follows that the property WC($i)$ is automatically inherited by all algebraic 
field extensions.
\\ \\
Example 1.2.2: A PAC field is WC(0).  In particular, separably closed $\implies$ WC(0).\footnote{For more about the PAC property, see the 
Appendix.} 
\begin{thm}(Lang, \cite{Lang56})
A finite field is WC(0).
\end{thm}
\noindent
Geyer and Jarden have constructed many non-PAC WC(0) fields \cite{GeyerJarden}.
\\ \\
Example 1.2.3: The field $k = \R$ is not WC(0): e.g. the smooth model of
\[y^2 = -(x^4+1) \]
is a genus one curve without an $\R$-point.  On the other 
hand, since $\gg_{\R} = \Z/2\Z$, every Galois cohomology group $H^i(\R,M)$ 
with $i > 0$ is $2$-torsion.  In particular, any nontrivial torsor under a 
real abelian variety has period equals index equals $2$, so $\R$ is certainly 
WC(1). 
\\ \\
In view of the local transition theorems for properties C$_i$ and Br(i), one might guess that 
if $K$ is a complete discretely valued field with WC(i) residue field $k$, then $K$ is itself 
WC(i+1), at least in equal characteristic.  However, this is generally very far from being the case, and 
is not even true in the (most favorable) case in which $k$ is algebraically closed of characteristic $0$.
\\ \\
For $g \in \Z^+$, let us write $\C_g$ for the iterated Laurent series field 
$\C((t_1))\cdots ((t_{g}))$.\footnote{Here we are, of course, using $\C$ to denote the field of complex numbers, 
although an arbitrary algebraically closed field of characteristic $0$ would serve as well.}  In particular $\C_1 = \C((t))$.  
The field $\C_g$ has cohomological dimension $g$ at every prime $p$, has 
property C$_{g}$ but not C$_{g-1}$, and has property 
Br(g-1).  The absolute Galois group of $\C_g$ is isomorphic to $\hat{\Z}^g$. 
\\ \\
The following two classical results limit the WC(i) properties of these fields:
\begin{prop}(Lang-Tate \cite[p. 678]{LT})
\label{LTPROP}
Let $g,n \in \Z^+$ with $n > 1$. \\
a) Let $E_{/\C_g}$ be any elliptic curve with $j(E) \in \C$. \\  Then there 
exists a torsor $\eta \in H^1(\C_g,E^g)$ with $P(\eta) = n, \ I(\eta) = n^g$.  \\
b) Let $E_{/\C_{2g}}$ be any elliptic curve with $j(E) \in \C$. \\  Then there 
exists a torsor $\eta \in H^1(\C_{2g},E^g)$ with $P(\eta) = n, \ I(\eta) = n^{2g}$.
\end{prop}
\noindent
%
\begin{thm}(Shafarevich \cite{Shafarevich61})
\label{SHATHM}
For every $n > 1$, there exists an abelian variety $A_{/\C((t))}$ and a torsor 
$\eta \in H^1(\C((t)),A)$ such that $P(\eta) = n$, $I(\eta) > n$.
\end{thm}
\noindent
So $\C_1 = \C((t))$ is \emph{not} WC(1).

\subsection{Property Almost WC(i)} \textbf{} \\ \\
We do not know whether $\C((t))$ is WC(2).  However, it is at least ``very close'':
\begin{thm}
\label{CTTHM1}
There exists a function $f: \Z^+ \ra \Z^+$ such that: for any $g \in \Z^+$, any $g$-dimensional abelian variety 
$A_{/\C((t))}$ and any torsor $\eta \in H^1(\C((t)),A)$, we have 
\[I(\eta) \leq f(g) P(\eta)^2. \]
\end{thm}
\noindent
This motivates the following moderate loosening of the WC(i) property:
\\ \\
A field $k$ is \textbf{almost WC(i)} if there exists a function $f(g)$ such that: for all finite extensions $l/k$, 
all abelian varieties $A_{/l}$ and all $\eta \in H^1(l,A)$, we have 
\[I(\eta) \leq f(\dim A) P(\eta)^i. \]
Thus Theorem \ref{CTTHM1} asserts that $\C_1 = \C((t))$ is almost WC(2).  Moreover, Proposition \ref{LTPROP} shows that 
$\C((t))$ is \emph{not} almost WC(0), whereas Theorem \ref{SHATHM} does not rule out the possibility that 
$\C((t))$ is almost WC(1).  
\\ \\
More generally, we shall prove:
\begin{thm}
\label{CTTHM2}
For all $i \in \Z^+$, the field $\C_i$ is almost WC(i+1).
\end{thm}
\noindent
Theorems \ref{CTTHM1} and \ref{CTTHM2} are almost immediate consequences of 
the proof of the Main Theorem, so we give their proofs at the end of $\S 3$.
\\ \\
In the other direction, Proposition \ref{LTPROP} shows that $\C_i$ is not 
almost WC(i-1).

\subsection{Property (almost) weakly (pp) WC(i)} \textbf{} \\ \\ \noindent
The situation is much different for fields of mixed characteristic. 
\begin{prop}
\label{PADICPROP1}
Let $K/\Q_p$ be an algebraic extension such that $K$ has finite degree over 
its maximal unramified subextension.  Then $K$ is not WC(g) for any $g \in \N$.
\end{prop}
\begin{proof} Let $g \in \N$.  In \cite[$\S 3.2$]{WCII} we construct a finite 
field extension $K_g/K$, an elliptic curve $E_{/K_g}$ and a torsor 
$\eta \in H^1(K_g,E^g)$ with $P(\eta) = p$, $I(\eta) = p^{g+1}$.  Thus 
the field $K_g$ is not WC(g).  The conclusion follows by applying Remark 1.2.1.
\end{proof}
\noindent
Nevertheless there are certainly nontrivial results on the period-index 
problem in WC-groups over $p$-adic fields, beginning with the celebrated 
theorem of Lichtenbaum that period equals index for genus one curves over 
a $p$-adic field.  One of the main results of \cite{WCII} is the following 
generalization:
\begin{thm}
\label{WCIITHM2}
Let $A_{/K}$ be a principally polarized abelian variety over a $p$-adic field.  
For $n \in \Z^+$, let $\eta \in H^1(K,A)[n]$.  Assume that \emph{at least} one of the 
following holds: \\
(i) $\dim A = 1$. \\
(ii) $n$ is odd. \\
(iii) There is a $\gg_K$-module isomorphism $A[n] \cong H \oplus H^*$, 
where $H, H^*$ are Cartier dual $\gg_K$-modules, each isotropic for the Weil 
pairing.  \\
Then $I(\eta) \leq (g!)\ P(\eta)^g$.
\end{thm}
\begin{proof}
This is \cite[Thm. 2]{WCII}.
\end{proof}
\noindent
If $K$ is a sufficiently large $p$-adic field, then the proof of 
Proposition \ref{PADICPROP1} constructs torsors $\eta$ under abelian varieties 
$A_{/K}$ of dimension $g$, period $p$ and index $p^g$.  Thus the bound of 
Theorem \ref{WCIITHM2} is sharp up to a multiplicative constant.  
\\ \\
This motivates the following definitions:
\\ \\
We say that $K$ is \textbf{weakly WC(i)} (resp. \textbf{weakly pp WC(i)}) if: 
for all finite extensions $L/K$, all nontrivial abelian varieties $A_{/L}$ 
(resp. all nontrivial principally polarized abelian varieties $A_{/L}$) 
and all classes $\eta \in H^1(L,A)$, we have 
\[I(\eta) \leq P(\eta)^{\dim A + i -1}. \]
\\ 
In particular, in any weakly pp WC(1) field $k$, the analogue of Lichtenbaum's theorem 
holds: period equals index for all genus one curves over all finite extensions of 
$k$. 
\\ \\
Finally, we say that $k$ is \textbf{almost weakly WC(i)} (resp. \textbf{almost 
weakly pp WC(i)}) if there exists a function $C(g)$ such that for all finite 
extensions $l/k$, all nontrivial abelian varieties $A_{/L}$ (resp. all 
nontrivial pp abelian varieties $A_{/L}$) and all classes $\eta \in H^1(L,A)$ 
we have 
\[I(\eta) \leq C(g) P^{\dim A + i -1} . \]
With this terminology, Theorem \ref{WCIITHM2} \emph{nearly} asserts that a 
$p$-adic field is almost weakly pp WC(1), the drawback being the requirement of 
at least one of the additional hypotheses (i) through (iii).  This 
drawback is removed by our Main Theorem.
 
\subsection{Restatement of the Main Theorem} \textbf{} \\ \\
\noindent
Let us now restate our main result using the language of the previous sections.%
\begin{nolabelthm}
\label{THM1}
Let $K$ be a CDVF with almost weakly WC(i) 
residue field $k$.  \\
a) If $k$ is moreover almost WC(i), then $K$ is almost weakly pp WC(i+1). \\
b) If $\car(k) = 0$, then $K$ is almost weakly pp WC(i+1).
\end{nolabelthm}
\noindent

\subsection{Examples and remarks}
\textbf{} \\ \\
Remark 1.6.1: The hypothesis on the perfection of $k$ is essential.  Indeed, in \cite[Remark 9.4]{LLR} the authors 
construct a CDVF field $K$ whose residue field is imperfect but separably closed (and hence WC(0) -- c.f. the Appendix) 
of characteristic $p > 0$ and, for all integers $0 < r \leq n$, a genus one curve over $K$ of period $p^n$ and 
index $p^{n+r}$.  Thus $K$ is \emph{not} almost weakly pp WC(1).  
\\ \\
Example 1.6.2: Since $\R$ is almost WC(0), our Main Theorem shows that $K = \R((t))$ to show that $K$ is almost 
WC$(1)$.  Other hand, $K$ is \emph{not} WC$(1)$: there exists a genus one curve $C_{/K}$ of period $2$ and index $4$ 
\cite{WCIII}.
\\ \\
We saw above that a $p$-adic field has WC(0) residue field and is not almost WC(i) for any i.  Here is a 
similar example in equicharacteristic $0$:
\\ \\
Example 1.6.3: Let $k$ be a Hilbertian PAC field of characteristic $0$ \cite[Thm. 18.10.3]{FA}.  Then $k$ is WC(0), 
but for all $P > 1$, $\Hom(\gg_k,\Z/P \Z)$ is infinite.  So for every $g \in \Z^+$ and $P > 1$, there exists a Tate 
elliptic curve $E_{/K}$ and a torsor $X \in H^1(K,E^g)$ with period $P$ and index $P^g$. 
\\ \\
By a famous theorem of Lang, any finite field $k$ is WC(0) \cite{Lang56}.\footnote{See the appendix for further discussion 
and a connection to an earlier result of F.K. Schmidt.}  Thus the Main Theorem asserts that any \textbf{local field} -- 
i.e., a finite extension of $\Q_p$ or of $\F_p((t))$ -- is weakly almost pp WC(1).  As mentioned above, this should be 
compared with Theorem \ref{WCIITHM2}: in this case the Main Theorem applies to all torsors under abelian varieties at the 
cost of a larger function $C(g)$.  
\\ \\
Especially, let us take $g = 1$ and compare with the theorems of Lichtenbaum and Milne.  Our proof yields $I \ | \ 192 P$
for genus one curves over a field $K$ with perfect WC$(0)$-residue field.  In particular $(6,P) = 1 \implies P = I$.  
Working a little more carefully and adding the hypothesis that $\gg_k$ is procyclic, one 
finds easily that $I \leq 16 P$ and $I \ | \ 48 P$.  This leaves open the question of whether 
the full Lichtenbaum theorem extends to our setting or whether the use of Tate duality is essential.  We hope to return to this point in a future work.
\\ \\
Another noteworthy feature of our Main Theorem is that, in contrast to the work of \cite{Lichtenbaum}, \cite{WCII}, \cite{WCII}, torsors of 
period divisible by $\car(K)$ are \emph{not} ruled out.  However, Milne proved that period equals index for 
genus one curves over $\F_q((t))$ by establishing an analogue of Tate local duality in flat cohomology.  Again, we recover 
Milne's theorem up to a constant in a general setting to which Tate duality seems inapplicable.  Our new technique is to 
generalize the period-index obstruction map $\Delta: H^1(K,A[P]) \ra \Br(K)$ to the setting of flat cohomology on $\Spec K$.  Fortunately for us, 
we need only quite formal properties of $\Delta$ whose proofs are direct analogues of the corresponding ones in the \'etale 
case.  In fact we also have work in progress on the explicit computation of $\Delta$ in terms of \textbf{symbol algebras} as is 
done in \cite{WCI}, \cite{WCII} when $A[P]$ is an \'etale group scheme whose corresponding Galois representation has sufficiently 
small image: \cite{WCV}.

\section{The period-index obstruction map in flat cohomology}
\noindent
Let $K$ be an arbitrary field, $A_{/K}$ an abelian variety of dimension $g$ and $P \in \Z^+$.  Then the morphism 
$[P]: A \ra A$ is an isogeny: in particular, on $\overline{K}$-points we have a short exact sequence 
\begin{equation}
\label{KUMMER0}
0 \ra A[P](\overline{K}) \ra A(\overline{K}) \stackrel{[P]}{\ra} A(\overline{K}) \ra 0. 
\end{equation}
We recall the following basic fact \cite[$\S 18$]{AV}:
\begin{prop}
\label{RECALLPROP}
The following are equivalent: \\
(i) The isogeny $[P]: A \ra A$ is separable. \\
(ii) The finite $K$-group scheme $A[P]$ is \'etale. \\
(iii) The characteristic of $k$ does \emph{not} divide $P$.
\end{prop}
\noindent
When the equivalent conditions of Proposition \ref{RECALLPROP} hold, $[P]: A \ra A$ 
is an \'etale covering, so the fiber over any point $Q \in A(K^{\sep})$ is a finite \'etale algebra.  It follows 
that the sequence 
\begin{equation}
\label{KUMMER1}
0 \ra A[P](K^{\sep}) \ra A(K^{\sep}) \stackrel{[P]}{\ra} A(K^{\sep}) \ra 0 
\end{equation}
is exact.  
\subsection{Perfect fields of positive characteristic}
\textbf{} \\ \\ \noindent
Suppose that $K$ is a perfect field of characteristic $p > 0$, and let $P = p^k$, $k \in \Z^+$.  
In this case (\ref{KUMMER0}) and (\ref{KUMMER1}) are one and the same, so that (\ref{KUMMER1}) is exact.  Letting 
$A[P]^{\circ}$ denote the maximal 
\'etale quotient of $A[P]$, we may reinterpret (\ref{KUMMER1}) as a short exact sequence of abelian sheaves on the \'etale site of $\Spec K$
\[0 \ra A[P]^{\circ} \ra A \stackrel{[P]}{\ra} A \ra 0, \]
and taking \'etale = Galois cohomology we get 
\[0 \ra A(K)/PA(K) \ra H^1(K, A[P]^{\circ}) \ra H^1(K,A)[P] \ra 0. \]
\\
Example 2.2.1: Let $E_{/K}$ be a supersingular elliptic curve.  Then $E[P]^{\circ} = 0$, so that every torsor under $E$ 
has a $K$-rational point.
\\ \\
Returning to the general case, let $\eta \in H^1(K,A)[P]$ be any torsor under $A$, and choose any lift of $\eta$ to $\xi \in H^1(K,A[P]^{\circ})$.  By 
\cite[Prop. 12]{WCII}, $\xi$ can be split by an extension of degree at most $\# A[P]^{\circ}$ and has index dividing 
$\# A[P]^{\circ}$.  But $\# A[P]^{\circ} \ | \ P^g$, with equality iff $A$ is \emph{ordinary}.  
Since $\xi$ splits $\implies$ $\eta$ splits, we have the following result.
\begin{prop}
\label{PERFECTPROP}
Let $K$ be a perfect field of positive characteristic $p$, let $P$ be a power of $p$ and let $\eta \in H^1(K,A)[P]$.  Then 
$I(\eta) \ |  \ P(\eta)^g$.
\end{prop}
\noindent
Remark 2.2.2: A similar argument applies to the ``usual'' Kummer sequence 
\[1 \ra \mu_p(\overline{K}) \ra \Gm(\overline{K}) \stackrel{[p]}{\ra} \Gm(\overline{K}) \ra 1, \]
to give that
\[\Br(K)[p] = H^2(K,(\mu_p)^{\circ}) = H^2(K,0) = 0. \]
\subsection{The Kummer sequence in flat cohomology}
\textbf{} \\ \\
We return to the case of general $K$ and $P$.  In this case we have a short exact sequence of commutative 
algebraic $K$-group schemes 
\[0 \ra A[P] \ra A \stackrel{[P]}{\ra} A \ra 0. \]
Viewing this as a short exact sequence of abelian sheaves on the flat site of $K$ and taking flat cohomology we get 
\[0 \ra A(K)/PA(K) \ra H^1(K,A[P]) \ra H^1(K,A)[P] \ra 0. \]
As an application of this formalism, we shall prove:
\begin{thm}
\label{GENERALPIBOUNDTHM}
Let $K$ be a field, $A_{/K}$ a $g$-dimensional abelian variety, and $\eta \in H^1(K,A)$ a torsor.  Then 
$I(\eta) \ | \ P(\eta)^{2g}$. 
\end{thm}
\noindent
Remark 2.3.1: Special cases of this result are due to Lang-Tate \cite[p. 678]{LT}, Lichtenbaum 
\cite[Thm. 8]{Lichtenbaum70} and Harase \cite[Thm. 4]{Harase}. 
\begin{proof}
We will show that any class $\xi \in H^1(K,A[P])$ splits over a field of degree at most $P^{2g}$.  By 
the surjectivity of $H^1(K,A[P]) \ra H^1(K,A)[P]$, this implies that any $\eta \in H^1(K,A)[P]$ 
has a splitting field of degree at most $P^{2g}$.  By an easy primary decomposition argument as in \cite[Prop. 12]{WCII}, 
we conclude that $I(\eta) \ | \ P^{2g}$.  \\ \indent
To establish this, consider the $K$-group scheme $A[P]$: it is the group scheme of all 
automorphisms $\varphi$ of $A$ commuting with $[P]: A \ra A$: $[P] = [P] \circ \varphi$.  Therefore, by the principle 
of descent, the flat cohomology group $H^1(K,A[P])$ classifies $(\overline{K}/K$)-twisted forms of $[P]: A \ra A$, 
i.e., morphisms $q: X \ra A$ fitting into a diagram

\[\begin{CD}
X_{/\overline{K}}  @>\sim>> A_{/\overline{K}}  \\
q @VVV          @V[P]VV   \\
 A_{/\overline{K}} @>1>> A_{/\overline{K}}
\end{CD} \]
In particular, $X_{/K}$ is a torsor under $A$ and the map $q: X \ra A$ has degree equal to the degree of $[P]: A \ra A$, 
namely $P^{2g}$.  Thus $q^*([O])$ is an effective $K$-rational zero-cycle on $X$ of degree $P^{2g}$, so yields a 
closed point $Q$ on $X$ such that $[K(Q):K] \leq P^{2g}$.  Over the field extension $K(Q)$, the morphism 
$q: (X,Q) \ra (A,O)$ and is therefore isomorphic to $[n]_{/K(Q)}: A \ra A$.  It follows that $K(Q)$ is a splitting 
field for $\xi$, completing the proof.
\end{proof} 
\noindent
Remark 2.3.2: In the earlier literature on the subject, together with the index one finds the \textbf{separable 
index}, the least positive degree of a $K$-rational divisor with support in $K^{\sep}$.  Our proof of Theorem 
\ref{GENERALPIBOUNDTHM} does \emph{not} give an upper bound for the separable index.  However, the recent preprint 
\cite{LG} shows that the index is equal to the separable index for smooth varieties over any field.  In particular this 
applies to torsors under abelian varieties and gives that the separable index divides the $(2g)$th power of the period.
\begin{cor}
\label{STRONGPOLCOR}
Let $(A,\lambda)_{/K}$ be a principally polarized abelian variety.  Then there exists a field extension $L/K$ of degree at most 
$2^{2 \dim A}$ such that $(A,\lambda)_{/L}$ is \emph{strongly} principally polarized.
\end{cor}
\begin{proof}
As recalled in $\S 0.1$, the obstruction to $\lambda$ being strong is an element $\Phi_{\PS}(\lambda) \in H^1(K,A)[2]$.  
Now apply Theorem \ref{GENERALPIBOUNDTHM}.
\end{proof}

\subsection{The period-index obstruction in flat cohomology} 
\textbf{} \\ \\ \noindent
Let $(A,\lambda)_{/K}$ be a strongly principally polarized abelian variety over an arbitrary field, and let $P \geq 2$ 
be an integer.  Let $\mathcal{L}$ be the $P$th multiple of the principal polarization.  Then D. Mumford has defined a \textbf{theta group}
\begin{equation}
\label{THETAEQ}
0 \ra \Gm \ra \mathcal{G}_{\mathcal{L}} \ra A[P] \ra 0, 
\end{equation}
in which $\Gm$ is the exact center \cite[$\S$ 23]{AV}.  When $\car(k) \nmid P$, (\ref{THETAEQ}) can be viewed 
as a sequence of sheaves of groups on the \'etale site of $K$.  Because $\Gm$ is central in $\mathcal{G}_{\mathcal{L}}$, there is a 
connecting map in nonabelian Galois cohomology 
\[\Delta: H^1_{\et}(K,A[P]) \ra H^2_{\et}(K,\Gm). \]
The map $\Delta$ was first studied by C.H. O'Neil in the case of $\dim A = 1$.  She had the fundamental insight that 
$\Delta$ is intimately related to the period-index problem, and accordingly she named $\Delta$ the \textbf{period-index 
obstruction map}.
\\ \\
Now we define the period-index obstruction map in arbitrary characteristic.  To do this, it suffices 
to identify (\ref{THETAEQ}) as a central, short exact sequence of sheaves of groups on the flat site of $\Spec K$.  We 
then get a connecting homomorphism 
\[ \Delta: H^1_{\fl}(K,A[P]) \ra H^2_{\fl}(K,\Gm) = H^2_{\et}(K,\Gm) = \Br(K); \]
again we have used the equality of flat and \'etale cohomology for smooth $K$-group schemes.  Moreover, having already 
introduced the notation $H^*(K,(A[P])^{\circ})$ for the \'etale cohomology of the maximal \'etale 
quotient of $A[P]$, it seems completely safe to abbreviate $H^*_{\fl}(K,A[P])$ to $H^*(K,A[P])$, and we do so in the 
sequel.  
\\ \\
The relation between $\Delta$ and the period-index problem is as follows.  Let $\eta \in H^1(K,A)[P]$, and let $\xi$ 
be any lift of $\eta$ to $H^1(K,A[P])$.  We again apply the principle 
of descent to characterize $H^1(K,A[P])$ as parameterizing a set of twisted forms of $A$ endowed with extra structure.  
This time our fundamental object is the morphism into projective space determined by the ample, basepoint free line 
bundle $\mathcal{L}$: $\varphi: A \ra \PP^N$.  It follows from the definition of the theta group scheme 
that $A[P]$ is precisely the group of translations $\tau_x$ of $A$ which extend via $\varphi$ to linear automorphisms 
$\gamma(\tau_x)$ of $\PP^N$, i.e., which render the following diagram commutative:
\[
\begin{CD}
A @>{\tau_x}>> A \\
@V{\varphi}VV @VV{\varphi}V \\
\PP^{N} @>{\gamma(\tau_x)}>> \PP^{N}
\end{CD} 
\]
\noindent
The twisted forms of $\varphi$ are morphisms $X \ra V$, where $X$ is a torsor under $A$ and $V$ is a Severi-Brauer variety.  
There is therefore a corresponding class $[V]$ in the Brauer group of $K$.  In \cite[$\S 5$]{WCII} it is shown that 
$\Delta(X \ra V) = [V]$.  The key step is that $\Delta: H^1(K,A[P]) \ra H^2(K,\Gm)$ factors through $H^1_{\fl}(K,\PGL_N) = 
H^1_{\et}(K,\PGL_N)$ (smoothness again).  The map $H^1(K,\PGL_N) \ra H^2(K,\Gm)$ can 
be computed explicitly in terms of cocycles as in $\S 5$ of \emph{loc. cit.} 
\\ \indent
This has the following consequence: suppose that $\eta \in H^1(K,A)$ is a class for which there exists at least one Kummer 
lift of $\eta$ to $\xi \in H^1(K,A[P])$ such that $\Delta(\xi) = 0$.  Then $\xi$ corresponds to a morphism $f: X \ra \PP^N$ 
of degree equal to the degree of $\varphi_{\mathcal{L}}$, namely $(g!) P^g$.  Intersecting the image of $f$ with a 
suitable linear subvariety of $\PP^N$ and pulling back via $f$, we get a $K$-rational zero cycle on $X$ of degree 
$g! P^g$.  We conclude that $I(\eta) \leq (g!) \ P^g$.  
%
\begin{thm}
The period-index obstruction map is a quadratic map on abelian groups: the associated map 
\[B: H^1(K,A[P]) \times H^1(K,A[P]) \ra \Br(K), \ (x,y) \mapsto \Delta(x+y) - \Delta(x) - \Delta(y) \]
is bilinear.
\end{thm}
\begin{proof}
Again, the key is that $\Delta: H^1(K,A[P]) \ra H^2(K,\Gm)$ factors through $H^1(K,\PGL_N)$, so $\Delta$ satisfies the 
same formal properties as the connecting homomorphism $\Delta'$ for the central short exact sequence of smooth group schemes
\[1 \ra \Gm \ra GL_N \ra \PGL_N \ra 1. \]
Finally, it follows from a theorem of Zarhin that $\Delta'$ is a quadratic map \cite{Zarhin}.
\end{proof}

\begin{cor}
\label{BRCOR}
Define $P^*$ to be $P$ if $P$ is odd and $2P$ if $P$ is even.  Then
\[\Delta(H^1(K,A[P])) \subset \Br(K)[P^*]. \]
\end{cor}
\begin{proof}
This holds for any quadratic map between abelian groups: \cite[$\S 6.6$]{WCII}.  
\end{proof}
\noindent
This has as a consequence the following relation between Brauer groups and WC-groups, a 
characteristic-unrestricted version of \cite[Thm. 6]{WCII}.

\begin{thm}
\label{BRTHM}
Suppose that a field $K$ is Br(d) for some $d \in \N$.  Then $K$ is almost weakly pp WC(d+1).
\end{thm}
\begin{proof} 
Let $L/K$ be a finite extension, and let $(A,\lambda)_{/L}$ be a principally polarized abelian variety of dimension $g$.  By 
Corollary \ref{STRONGPOLCOR}, up to replacing $L$ by an extension field of degree at most $2^{2 g}$, we may assume 
that the polarization is strong.\footnote{We will need to use this observation several times in the sequel.  We will not further 
belabor the point but just include a correction factor of $2^{2 \dim A}$ whenever we make use of the period-index 
obstruction map.}  Let $\eta \in H^1(L,A)$ be any class, and choose any lift of $\eta$ to $\xi \in H^1(L,A[P])$.  By Corollary \ref{BRCOR}, $\Delta(\xi) \in \Br(L)[2P]$, so by our Br(d) 
assumption, there exists a splitting field $M/L$ for $\Delta(\xi)$ of degree at most $2^d P^d$.  Since 
$\Delta(\xi|_M) = (\Delta \xi)|_M = 0$, it follows that there exists an $M$-rational zero-cycle on the corresponding 
torsor $X$ of degree $(g!) P^g$, so altogether $\eta$ has a splitting field of degree at most \[ 2^{2g} 2^d P^d \cdot (g! P^g) = 
(2^{2g+d} g!) P^{\dim A + (d+1) - 1}. \]
\end{proof}
\noindent
In the sequel, we will use the following special case: if $\Br(K) = 0$, then for any principally polarized abelian variety $A_{/K}$ and 
$\eta \in H^1(K,A)$, $I(\eta) \leq 2^{2g} (g!) \ P(\eta)^{\dim A}$.  



\section{Proof of the Main Theorem}


\subsection{Travaux de Gabber-Liu-Lorenzini, Liu-Lorenzini-Raynaud}
\textbf{} \\ \noindent
\begin{thm}(\cite[Prop. 8.1]{LLR})
\label{LLRTHM}
Let $K$ be a discretely valued field with valuation ring $R$ and perfect residue field $k$.  Let $A_{/K}$ be an 
abelian variety with good reduction.  Then any torsor $V$ under $A$ admits a 
proper regular model $X_{/R}$ endowed with an action $A \times_R X \ra X$ 
extending the structure of a torsor under $A_{/K}$ on the generic fiber and 
such that the map $A \times_R X \ra X \times_R X$ given by $(a,x) \mapsto (ax,x)$ 
is surjective.  Then the reduced subscheme $X^{\red}_{/k}$ of the special fiber is 
a torsor under an abelian variety isogenous to $A_{/k}$.  
\end{thm}
\noindent
We will call a model $X_{/R}$ as in the statement of Theorem \ref{LLRTHM} an \textbf{LLR model}.
\begin{thm}(Index Specialization Theorem \cite{GLL})
\label{GLLTHM}
Let $K$ be a Henselian discretely valued field with valuation ring $R$ and residue field $k$.  Let $X$ be regular 
scheme equipped with a proper flat morphism $X \ra \Spec R$.  Denote by $X_{/K}$ (resp. $X_{/k}$) the generic (resp. special) 
fiber of $X \ra \Spec R$.  Write $X_k$ as $\sum_{i=1}^n r_i \Gamma_i$, where each $\Gamma_i$ is irreducible and of 
multiplicity $r_i$ in $X_{/k}$.  Then 
\[ I(X_{/K}) = \gcd_i r_i I(\Gamma_i^{\reg}). \]
\end{thm}

\subsection{Beginning of the proof} \textbf{} \\ \\ \noindent
Suppose $K$ is complete, discretely valued field with perfect residue field $k$.  If $\car(k) = 0$, 
we are assuming the existence of $i \in \N$ and a function $c: \Z^+ \ra \Z^+$ such that: for all 
finite extensions $l/k$, all nontrivial abelian varieties $A_{/l}$ and all torsors $\eta \in H^1(l,A)$, we have 
\[I(\eta) \leq c(\dim A) P(\eta)^{\dim A + i - 1}. \]
If $\car(k) > 0$, we assume that there exists $i \in \N$ and a function $c: \Z^+ \ra \Z^+$ such that for all abelian varieties
$A_{/k}$ and all torsors $\eta \in H^1(k,A)$, 
\[ I(\eta) \leq c(\dim A) P(\eta)^i . \]
In either case we wish to show that there exists a function $C: \Z^+ \ra \Z^+$ such that for all finite extensions $L/K$, 
all principally polarizable abelian varieties $A_{/L}$ and all torsors $\eta \in H^1(L,A)$, 
\[I(\eta) \leq C(\dim A) P(\eta)^{\dim A + i}. \]
Both the hypothesis and the conclusion are stable under finite base extensions, so it is no loss of generality to 
assume $L = K$.  Thus, let $A_{/K}$ be a principally polarizable abelian variety and $\eta \in H^1(K,A)$.  Let 
$(X\,mu)$ be the corresponding torsor under $A$.  By (common) abuse of notation we will omit the action $\mu$ in what follows.

\subsection{Good reduction} \textbf{} \\ \\
\noindent
Suppose $A_{/K}$ has good reduction.  
\\ \\
In the case when $k$ is WC(0) we can prove somewhat sharper results, so we begin there and then discuss 
modifications necessary to establish the general case.
\\ \\
Let $P_{\unr}$ (resp. $I_{\unr}$) denote the period (resp. the index) of the torsor $X$ extended to 
the field $K^{\unr}$.  As for every field extension, we have $P_{\unr} \ | \ P$ and 
$I_{\unr} \ | \ I$.  By Lang's theorem $\Br(K^{\unr}) = 0$, so by Theorem \ref{BRTHM}, $I_{\unr} \leq 2^{2g} (g!) \  P_{\unr}^g$.    
\\ \\
Therefore the following result implies the Main Theorem in this case.

\begin{prop}
\label{WC0PROP}
Suppose that $A$ has good reduction and $k$ is WC$(0)$.  Then
\[P = P_{\unr}, \ I = I_{\unr}. \]
\end{prop}
\begin{proof} 
Step 1: We claim that $V$ does not split in $K^{\unr}$.  (In other words, 
since $V$ is arbitrary, we claim that the relative WC-group $H^1(K^{\unr}/K,A)$ is 
trivial in the case of good reduction.)  Indeed, suppose to the contrary that $X$ admits a $K^{\unr}$-rational 
point.  Let $A_{/R}$ be the N\'eron model of $A$ and $X_{/R}$ an LLR model.  It follows that the special fiber of $X$ is smooth.  
Now we use our WC$(0)$: $X(k) \neq \varnothing$.    By Hensel's Lemma, 
$V(K) \neq \varnothing$.  
\\ \\
Step 2: Now consider the class $\eta' = P_{\unr} \eta$ in $H^1(K,A)$.  
We have 
\[\eta' |_{K^{\unr}} = P_{\unr} \eta |_{K^{\unr}} = P_{\unr} (\eta_{K^{\unr}}) = 0, \]
so by Step 1 $\eta' = 0$.  It follows that $P \ | \ P_{\unr}$, so $P = P_{\unr}$.
\\ \\  
Step 3: We appeal to the LLR-model $X_{/R}$ of Theorem \ref{LLRTHM}.  The condition 
WC(0) implies that $X^{\red}_{/k}$ has a $k$-rational point.  By the Index Specialization Theorem, the index of $X$ is equal to the multiplicity of the special fiber.  This quantity does not 
change upon unramified base extension, so $I = I_{\unr}$.  This completes the proof of Proposition \ref{WC0PROP}, 
and with it the good reduction case of Theorem \ref{THM1}a).
\end{proof}
\noindent
Virtually the same argument works if $k$ \emph{almost} WC$(0)$: i.e., if there exists a constant 
$c(g)$ such that for all principally polarized $A_{/k}$ and $\eta \in H^1(k,A)$, $I(\eta) \leq c(g)$.  Let 
$X$ be the torsor corresponding to $\eta$ and $X_{/R}$ its LLR-model. By assumption, after making a field 
extension of degree at most $c(g)$, the reduced special fiber has a rational point.  Again by index specialization, 
it follows that we can further trivialize the class by making a totally ramified extension of degree $I_{\unr} 
\leq 2^{2g} g! P_{\unr}^g \leq 2^{2g} g! P^g$.  Thus overall $V$ can be split by an extension of degree at most $ 2^{2g} (g!) \ c(g)\  P^{g}$.
\\ \\
Next we recall the following closely related result, which is essentially due to Lang and Tate 
(c.f. \cite[Cor. 1]{LT}).
\begin{cor}
\label{LTCOR}
Let $K$ be a CDVF with perfect WC$(0)$-residue field $k$.  Let $A_{/K}$ an abelian variety and 
$\eta \in H^1(K,A)$.  If $\car(k)$ does not divide $P(\eta)$, then a finite field extension $L/K$ splits $X$ 
if and only if $e(L/K) \ | \ P(\eta)$.  
\end{cor}
\begin{proof}
Let $L/K$ be a finite extension splitting $V$.  As for any finite extension of a local field, we can decompose it 
into a tower $L/M/K$, where $M/K$ is unramified and $L/M$ is totally ramified.  By our previous results, we know 
that the unramified extension $M/K$ is index-nonreducing: $I(V_{/L}) = I = P$.  So certainly $P$ divides $[L:M] = e(L/K)$.  
Conversely, suppose $L/K$ is a finite extension with $P \ | \ e(L/K)$; we wish to show that $L$ splits $V$.  Decomposing 
$L/K$ into $L/M/K$ as above, it is enough to show that $L$ splits $V_{/M}$, i.e., we reduce to the case in which 
$L/K$ is totally tamely ramified (ttr).  By the structure of ttr extensions, there exists a unique degree 
$P$ subextension, so we may further assume that $[L:K] = P$.  But we know that there is at least one degree $P$ ttr 
splitting field, and although there are in general several ttr extensions of degree $P$, the compositum of any two of 
them with $K^{\unr}$ coincide.  It follows that each such $L/K$ is a splitting field.
\end{proof}
\noindent
Remark 3.2.1: In the case in which $k = \overline{k}$, we can say even more: by the criterion of
N\'eron-Ogg-Shafarevich, we have $A[P] = A[P](K)$, and $A(K)$ is
$P$-divisible, so the Kummer sequence trivializes to give isomorphisms
\[H^1(K,A)[P] = \Hom(\gK,A[P]) \cong \Hom(\gk,\Z/P\Z)^{2g} \]
which are functorial with respect to restriction to
finite extensions.  Since tame ramification
groups are cyclic, we get $\Hom(\gK,\Z/p\Z) \cong \Z/P\Z$, i.e.,
every element of $H^1(K,A)[P]$ is split by $k((t^{\frac{1}{P}}))$, the unique
degree $P$ extension of $K$.  
\\ \\
Next, suppose that $k$ is almost WC(i).  Arguing as above, the Index Specialization Theorem gives us 
\begin{equation}
\label{IREDEQ}
I(\eta) = I_{\unr}(\eta) I(X^{\red}_{/k}). 
\end{equation}
Combining the estimate $I_{\unr}(\eta) \leq 2^{2g} (g!) \ P(\eta)^{\unr}$ of Theorem \ref{BRTHM} and 
\[I(X^{\red}_{/k}) \leq c(g) P(X^{\red}_{/k}))^i \leq c(g) P(\eta)^i, \]
we conclude 
\[I(\eta) \leq  2^{2g} (g!) \  c(g) \ P(\eta)^{g+i}, \]
establishing the Main Theorem in this case.  
\\ \\
Finally, suppose that $\car(k) = 0$.  Then by Corollary \ref{LTCOR}, 
\[I_{\unr}(\eta) = P_{\unr}(\eta) \ | \ P(\eta). \]
Substituting this into (\ref{IREDEQ}) and using our hypothesis that $I(X^{\red}_{/k}) \leq c(g) P(\eta)^{g+i-1}$, 
we conclude that 
\[I(\eta) \leq c(g) P(\eta)^{g+i}. \]

\subsection{Purely toric reduction} 
\begin{prop}(Gerritzen)
\label{GERRITZENPROP}
Let $K$ be any field.  Let $\tilde{A}$ be a $\mathfrak{g}_K$-module and
$\Gamma \subseteq \tilde{A}$ a $\mathfrak{g}_K$-submodule which is torsionfree as a $\Z$-module and such
that $\Gamma^{\gk} = \Gamma$; put $A := \tilde{A}/\Gamma$.
Suppose also that $H^1(L,\tilde{A}) = 0$
for all finite extensions $L/K$.  Let $\eta \in H^1(K,A)$ be a class of
period $P$.  Then: \\
a) $\eta$ has a \emph{unique} minimal splitting field $L = L(\eta)$. \\
b) The extension $L/K$ is abelian of exponent $P$. \\
c) $I(\eta) \ | \ P^g$, where $g = \dim_{\Q}(\Gamma \otimes \Q)$ is the rank of
$\Gamma$.
\end{prop}
\begin{proof}
See \cite{Gerritzen} or \cite[Prop. 16]{WCII}.
\end{proof}
\noindent
Now let $A_{/K}$ be a $g$-dimensional abelian variety with split toric reduction: the identity component 
of the N\'eron special fiber is $\Gm^g$.  Then $A$ admits an analytic uniformization: it is isomorphic, 
as a rigid $K$-analytic group, to $ \Gamma \backslash \Gm^g$, where
$\Gamma \cong \Z^g$ is a discrete subgroup.
\begin{cor}
Let $K$ be any complete, discretely valued field, and let $A_{/K}$ be a $g$-dimensional analytically uniformized abelian variety.  
For any $\eta \in H^1(K,A)$, $I(\eta) \ | \ P(\eta)^g$. 
\end{cor}
\begin{proof} 
We apply Proposition \ref{GERRITZENPROP} with $\tilde{A} = \Gm^g(K^{\sep})$.  By Hilbert 90, 
$H^1(L,\tilde{A}) = 0$ for all finite $L/K$ by Hilbert 90. The result follows immediately.
\end{proof}
%

\noindent
Now assume $A$ has purely toric reduction, not necessarily split.
Then there exists a function $F_1(g)$ such that the torus splits
over an extension of degree dividing $F_1(g)$.  Thus, at the cost
of possibly multiplying the ratio $I/P$ by $F_1(g)$, we can reduce
to the previous case, getting $I \ | \ F_1(g) \cdot P^g$.

\subsection{General case}

Now we recall the following result, due to Bosch and Xarles \cite{BX}, in the form which is proved in \cite{CX}:
\begin{thm}(Uniformization Theorem)
Let $A_{/K}$ be a $g$-dimensional abelian variety over a complete
field.  Then there exists a semiabelian variety $S_{/K}$ of
dimension $g$, whose abelian part has potentially good reduction,
a $\mathfrak{g}_K$-module $M$ whose underlying abelian group is
torsion free of rank equal to the toric rank of $S$ and an exact
sequence of $\mathfrak{g}_K$-modules
\begin{equation}
\label{ONE}
0 \ra M \ra S(K^{\sep}) \ra A(K^{\sep}) \ra 0.
\end{equation}
Moreover $\rank_{\Z}M^{\mathfrak{g}_K} = s$, the split toric rank
of the N\'eron special fiber of $A$.  For every finite extension
$L/K$, the \emph{identity components } of the N\'eron special
fibers of $S/L$ and $A_{/L}$ are isomorphic.
\end{thm}
\noindent
Moreover, by making a base extension of degree depending only 
on $g$, we can achieve split semistable reduction:
\begin{lemma}
For any positive integer $g$, there exists an integer $F(g)$ such
that for any $g$-dimensional abelian variety $A$ over a CDVF $K$,
there exists an extension $L/K$ of degree at most $F_1(g)$ such that
$A_{/L}$ has split semistable reduction.
\end{lemma}
\begin{proof} To get semistable reduction, we can trivialize
the Galois action on $A[3]$ (if the residue characteristic is not
$3$) or on $A[4]$ (if the residue characteristic is not $2$), and
to split the toric part of the reduction we need to trivialize the
Galois action on a rank $d \leq g$ finite free $\Z$-module, which
can be done over an extension of degree at most $\# GL_d(\Z/3\Z) \leq \# GL_g(\Z/3\Z)$.
Thus we could take
\[F_1(g) = \left ( \max \# GSp_{2g}(\Z/3\Z), \# GSp_{2g}(\Z/4\Z) \right
) \cdot \# GL_g(\Z/3\Z). \]
\end{proof}
\noindent
After making the reduction split semistable, the exact sequence (\ref{ONE}) 
above becomes

\begin{equation}
\label{TWO} 0 \ra \Z^{\mu} \ra S(K^{\sep}) \ra A(K^{\sep}) \ra 0, \end{equation} where
$S_{/K}$ is a semiabelian variety of the form
\begin{equation}
\label{THREE} 0 \ra \Gm^{\mu} \ra S \ra B \ra 0, \end{equation} and
$B$ is abelian with good reduction.  Taking
$\gK$-cohomology of (\ref{TWO}) gives
\[0 \ra H^1(K,S) \ra H^1(K,A) \stackrel{\delta}{\ra} H^2(K,\Z)^{\mu} \ra \ldots. \]
Moreover, taking $\gK$-cohomology of (\ref{THREE}) and applying Hilbert 90, we get an
injection \[ H^1(K,S) \hookrightarrow H^1(K,B). \]  Finally, we have \[H^2(K,\Z)^{\mu} =
\Hom(K,\Q/\Z)^{\mu}. \] 
So, starting with any class $\eta \in H^1(K,A)[P]$, put $\xi = \delta(\eta) \in H^2(K,\Z)$.  There is an abelian 
extension $L/K$ of exponent dividing $P$ and degree dividing $P^{\mu}$ which splits $\xi$, so that $\eta|_{L} \in 
H^1(L,S) \hookrightarrow H^1(L,B)$.  Since $B$ is an abelian variety of dimension $g - \mu$ with good reduction, 
the work of Section $\S 3.2$ applies to show that $\eta|_{L}$ can be split over an extension $M/L$ of degree at most 
$C(g-\mu) P^{g-\mu+i}$, so 
overall $\eta$ can be split by an extension $M$ of degree at most $C(g-\mu) F_1(g) P^{g+i} \leq C'(g) F_1(g) P^{g+i}$, 
where $C'(g) = \max_{1 \leq j \leq g} C(j)$.  This completes the proof of the Main Theorem.

\subsection{$K = \C_g$} \textbf{} \\ \\ 
\noindent
We now give the proof of Theorems \ref{CTTHM1} and \ref{CTTHM2}.
\\ \\
First, let $K = \C((t))$ and let $A_{/K}$ be an abelian variety of dimension $g$.  
In Corollary \ref{LTCOR} we established that if $A$ has good reduction, any 
class $\eta \in H^1(K,A)$ has period equals index.  At the other extreme, if $A$ has 
split multiplicative reduction, then $H^1(K,A)[P]$ injects into 
$X = \Hom(\mathfrak{g}_K,(\Z/P\Z)^g)$.  However, because 
$\mathfrak{g}_K \cong \hat{\Z}$ is procyclic, any element of $\Hom(\mathfrak{g}_K,(\Z/P\Z)^g)$ 
splits over a degree $P$ field extension, thus period equals index in this case 
as well.  Finally, after making a field extension of bounded degree $f(g)$ to 
attain split semistable reduction of $A$, the d\'evissage argument of $\S 3.4$ 
shows that we can split any class $\eta \in H^1(K,A)$ by first splitting its 
purely toric part and then splitting its good reduction part, getting overall a 
splitting field of degree at most $f(g) P^2$.  This completes the proof of 
theorem \ref{CTTHM1}.
\\ \\
Now let $K = \C_2 = \C((t_1))((t_2))$.  This time $\mathfrak{g}_K \cong \hat{\Z}^2$, 
so that if $A$ has split multiplicative reduction the sharp upper bound on 
the index is $P^2$.  Once again, after making a base extension of degree $f(g)$ 
to ensure split semistable reduction, the critical case is getting an upper 
bound on the index in the case of good reduction.  But now something extremely fortunate occurs: both the residue field $k$ and the maximal 
unramified extension $K^{\unr}$ are isomorphic to $\C_1 = \C((t))$!  So, by a 
now familiar argument with the Index Specialization Theorem, we can split a 
class of period $P$ by an extension of degree at most $f(2) P^2 \cdot P = f(2) P^3$.  
The general case follows by an obvious inductive argument.

\section{Appendix: Some Related Field Arithmetic}
\noindent
\subsection{PAC and I1 fields} 
\textbf{} \\ \\ \noindent
A field $k$ is \textbf{PAC} (pseudo-algebraically closed) if every geometrically integral 
$k$-variety has a $k$-rational point.  A field $k$ which admits a geometrically integral variety 
$V_{/k}$ of positive dimension with only finitely many $k$-rational points cannot be PAC: apply the definition 
to the complement of the set of rational points!  In particular a finite field is not PAC. 
\\ \\
Remark 4.1.1: Since every 
geometrically integral variety of positive dimension over an infinite field contains a geometrically integral 
curve \cite[Cor. 10.5.3]{FA}, it suffices to verify the PAC condition on geometrically integral curves.  
\\ \\
Example 4.1.2: All algebraically and separably closed fields are PAC.  An infinite algebraic extension of a finite 
field is PAC.  A nonprincipal ultraproduct of finite fields is PAC.  Certain large algebraic extensions of $\Q$ 
are PAC.  Any algebraic extension of a PAC field is PAC.  
\\ \\
A field $k$ has property \textbf{I1} if every geometrically irreducible variety over $k$ has a 
$k$-rational zero-cycle of degree one.  Again, it suffices to check this for curves.  
\\ \\
Example 4.1.3: Any PAC field is I1.  A theorem of F.K. Schmidt \cite{Schmidt} implies that every geometrically integral curve 
over a finite field $\F$ has a zero-cycle of degree $1$.  Applying the above remark about curves lying on 
higher dimensional varieties to the maximal pro-$p$ and maximal pro-$q$ extensions of $\F$ for primes 
$p \neq q$, one sees that $\F$ is I1.  
Example 4.1.4: Geyer and Jarden define a \textbf{weakly PAC} field to be a field $k$ such that for every 
geometrically integral variety $V$ such that $V_{/\overline{k}}$ is birational either to projective space 
(``Type 0'') or to an abelian variety (``Type 1'') already has a $K$-rational point.  Evidently weakly PAC 
implies WC$(0)$, so the following result gives many examples of WC$(0)$ fields:
\begin{prop}(Geyer-Jarden \cite[Lemma 2.5]{GeyerJarden})
\label{GJPROP}
For any countable field $k_0$, there exists a countable extension field $k$ with the following properties: \\
(i) $k$ is weakly PAC. \\
(ii) For every smooth curve $C_{/k_0}$ of genus at least $2$, $C(k) = C(k_0)$.  
\end{prop}
\noindent

\subsection{Two implications} \textbf{} \\ \\
It is well known to the experts that the property I1 implies the property 
Br$(0)$.  In this section we will factor this implication through the 
property WC$(0)$.
\begin{prop}
\label{LITTLEPROP}
Let $k$ be a I1 field, and let $G_{/k}$ be a connected commutative algebraic group.  Then $H^1(k,G) = 0$.  
In particular, $k$ is WC$(0)$.
\end{prop}
\begin{proof}
The elements of $H^1(k,G)$ correspond to torsors $X$ under the group $G$, and as in the case of WC-groups, the period of 
$X$ -- i.e., its order in the torsion group $H^1(k,G)$ -- divides its index, which is equal to the least positive degree 
of a zero-cycle on $X$.  The result follows immediately.
\end{proof}
\noindent
Remark 4.2.1: It seems to be unknown whether the converse holds, nor even whether all of 
the WC(0)-fields constructed by Geyer-Jarden have the I1 property. 
\begin{thm}
\label{WC0BR0}
A perfect WC$(0)$ field $k$ has the property Br$(0)$.  
\end{thm}
\begin{proof}
Step 0: It is enough to show that for any finite extension $k'/k$ and any prime number $\ell$, $\Br(k')[\ell] = 0$.  
However, by our hypotheses are stable under finite base change, so we may as well assume $k' = k$.  Further, let $p \geq 0$
be the characteristic of $k$.  In case $p > 0$, the assumed perfection of $k$ implies that $\Br(k)[p] = 0$ 
\cite[\S II.3.1]{CG}, and hence $\Br(k')[p] = 0$.  Thus we may assume that $\ell$ is different from the characteristic of 
$k$.
\\ \\
Step 1: We recall the following celebrated theorem of Merkurjev-Suslin: Let $k$ be a field of characteristic $p \geq 0$ and $n$ is a positive integer 
indivisible by $p$.  If $k$ contains the $n$th roots of unity, then $\Br(k)[n]$ is generated by the order $n$ norm-residue 
symbols $\langle a, b \rangle_n$ as $a,b$ run though $k^{\times}/k^{\times n}$ \cite{MS}.  
\\ \\
Step 2: We recall a consequence of the theory of the period-index obstruction map $\Delta$: let $\ell$ be a prime 
number, $M$ a field of characteristic not equal to $\ell$, and $E_{/M}$ an elliptic curve with full $\ell$-torsion rational 
over $M$: $\# E(M)[\ell] = \ell^2$.  (By the Galois-equivariance of the Weil pairing, this implies that 
$M$ contains the $\ell$th roots of unity.)  Then $H^1(M,E[\ell]) \cong (M^{\times}/M^{\times \ell})^2$, and 
for $\ell > 2$, the map $\Delta: H^1(M,E[\ell]) \ra \Br(M)$ is of the form $(a,b) \mapsto \langle C_1 a, C_2 b \rangle_{\ell} - 
\langle C_1, \ C_2 \rangle_{\ell}$ for suitable $C_1, C_2 \in K^{\times}$.  (More precise results are now known, but this weak version has the advantage of 
treating $\ell = 2$ and $\ell > 2$ uniformly, so is useful for our present purpose.)  It follows that the image of 
$\Delta$ contains every norm residue symbol, and therefore, by Merkurjev-Suslin, the subgroup generated by the image 
is all of $\Br(M)[\ell]$.  So, if for our particular elliptic curve $E_{/M}$ we have $H^1(M,E)[\ell] = 0$, it follows 
that $\Br(M)[\ell] = 0$.
\\ \\
Step 3: If $\ell = 2$, this strategy succeeds in a straightforward manner: over every field $k$ of characteristic different from 
$2$, there exists an elliptic curve with full $2$-torsion, namely $y^2 = x(x-1)(x-2)$.  So if every hyperelliptic quartic 
curve over $k$ has a rational point, every conic over $k$ has a rational point.
\\ \\
Step 4: If $\ell > 2$, we need not have an elliptic curve $E_{/k}$ with full $\ell$-torsion, since in particular we need not 
have the $\ell$th roots of unity rational over $k$.  So we employ a trick.  Let $E_{/k}$ be any elliptic curve with 
complex multiplication by any imaginary quadratic field.  Let $m = k(E[\ell])$.  Then $[m:k]$ divides
either $2(\ell^2-1)$ or $2(\ell-1)^2$, so in particular is prime to $\ell$.  It follows that the restriction map 
$\Br(k)[\ell] \ra \Br(m)[\ell]$ is injective, so it suffices to show that $\Br(m)[\ell] = 0$.  By hypothesis we have 
$H^1(m,E)[\ell] = 0$, so that applying Step 2 we conclude that $\Br(m)[\ell] = 0$ and hence $\Br(k)[\ell] = 0$.  
\end{proof}
\noindent
Example 4.2.2: Br$(0)$ does \emph{not} imply WC$(0)$.  This goes back to work of Ogg and Shafarevich: neither $\C(t)$ nor $\C((t))$ is a WC$(0)$ field.  In fact, if $A_{/\C((t))}$ is a $g$-dimensional abelian 
variety with good reduction, then consideration of the Kummer sequence
\[0 \ra A(k)/nA(k) \ra H^1(k,A[n]) \ra H^1(k,A)[n] \ra 0 \]
swiftly yields $H^1(k,A) \cong (\Q/\Z)^{2g}$.  
\\ \\
Remark 4.2.3: By a restriction of scalars argument, the proof easily gives that for any $\ell$ prime to the characteristic of 
$k$, if $\Br(k)[\ell] \neq 0$, then there exists a principally polarized abelian variety $A_{/k}$ of dimension at most 
$2(\ell^2-1)$ such that $H^1(k,A)[\ell] \neq 0$.  
\\ \\
Remark 4.2.4: It follows easily that weakly WC$(0)$ implies $\Br(k)[\ell] = 0$ for all sufficiently large primes $\ell$.  Indeed, weakly WC$(0)$ implies that there 
exists a fixed prime $\ell_0$ such that for all $\ell > \ell_0$ and all principally polarized abelian varieties $A_{/k}$ 
we have $H^1(k,A)[\ell] = 0$.  
\\ \\
Remark 4.2.5: It seems likely that the result is true without the assumption on the perfection of $k$.  The characteristic 
$p$ analogue of the Merkurjev-Suslin theorem is the (much earlier and easier) theorem of Teichm\"uller: for all $r \geq 1$, 
$\Br(k)[p^r]$ is generated by cyclic algebras \cite{Teichmuller}, \cite{GS}.  To make use of this we need to know the 
explicit form of the period-index obstruction map in the flat case, so we defer the issue to \cite{WCV}.
%
Example 4.2.6: Let $\F$ be a finite field.  By elementary arguments involving the zeta function, F.K. Schmidt 
showed that the index of any nice curve over a finite field is $1$.  It follows from Remark 4.1.1 that finite 
fields are I1.  Using Proposition \ref{LITTLEPROP} and Theorem \ref{WC0BR0} we deduce two more 
famous theorems: Lang's theorem that a torsor under a connected, commutative 
algebraic group has a rational point, and Wedderburn's theorem that the Brauer 
group of a finite field is trivial.
\\ \\
Next we recall the following theorem:
\begin{thm}(Steinberg, \cite{Steinberg}) Let $k$ be a perfect Br$(0)$ field and $G_{/k}$ a smooth, connected 
algebraic group.  Then $H^1(k,G) = 0$.
\end{thm}
\noindent

\begin{cor}
If $k$ is a perfect field such that $H^1(k,A) = 0$ for all abelian varieties 
$A/k$, then
$H^1(k,G) = 0$ for all connected algebraic groups $G_{/k}$.
\end{cor}
\begin{proof}
Recall that every connected algebraic group $G$ over a perfect
field $k$ admits a \textbf{Chevalley decomposition} \cite[Ch. IX]{BLR}: there is 
a (unique) normal linear subgroup $L$ of $G$ such that
$G/L = A$ is an abelian variety.  The result now follows ``by d\'evissage'', using the exact sequence of 
\cite[$\S$ I.5.5, Prop. 38]{CG}.
\end{proof}
\noindent
Finally, the proof of Theorem \ref{WC0BR0} leads naturally to the following question.
\begin{ques}
\label{QUES1}
Let $k$ be a field.  Characterize the set of all elements in the Brauer group which arise as the obstruction associated 
to a $k$-rational divisor class on some genus one curve $C_{/k}$.  Is it, for instance, the class of all cyclic algebras?
\end{ques}
\noindent
Remark 4.2.7: It is also possible to ask the question on the level of central simple algebras.  The Severi-Brauer variety 
associated to a degree $n \geq 2$ divisor class on a genus one curve $C$ is a twisted form of $\PP^{n-1}$, and the 
corresponding central simple algebra has degree $n-1$.  The problem can also be stated geometrically: find all 
Severi-Brauer varieties $V_{/K}$ such that there exists a genus one curve $C_{/K}$ and a morphism $C \ra V$.  

\end{document}